\documentclass[10pt,leqno]{article}

\usepackage{amsmath,amssymb,amsthm,mathrsfs,dsfont}
\usepackage{amssymb}
\usepackage{enumerate}
\makeatletter

\newcommand{\Rmnum}[1]{\expandafter\@slowromancap\romannumeral #1@}
\makeatother

\usepackage[margin=3cm]{geometry} 

\usepackage{titlesec,hyperref}

\usepackage{color}

\usepackage{fancyhdr}
\titleformat{\subsection}{\it}{\thesubsection.\enspace}{1pt}{}

\newtheorem{theo}{Theorem}[section]
\newtheorem{lemm}[theo]{Lemma}
\newtheorem{defi}[theo]{Definition}
\newtheorem{coro}[theo]{Corollary}
\newtheorem{prop}[theo]{Proposition}
\newtheorem{rema}[theo]{Remark}
\numberwithin{equation}{section}

\allowdisplaybreaks 

\begin{document}
\title{The local well-posedness, blow-up phenomena and ill-posedness of a new fifth-order Camassa-Holm type equation
	\hspace{-4mm}
}

\author{Yiyao $\mbox{Lian}^1$ \footnote{Email: lianyy7@mail2.sysu.edu.cn},\quad
	Zhaoyang $\mbox{Yin}^{1}$\footnote{E-mail: mcsyzy@mail.sysu.edu.cn}\\
$^1\mbox{School}$ of Science,\\ Shenzhen Campus of Sun Yat-sen University, Shenzhen 518107, China}

\date{}
\maketitle
\hrule

\begin{abstract}
	In this paper, we study a new fifth-order Camassa-Holm type equation derived by Li \cite{Li.Z}. We firstly establish the local well-posedness in the sense of Hadamard for the Cauchy problem of the new fifth-order Camassa-Holm type equation in Besov spaces. 
	Secondly, we obtain  blow-up criteria. Building upon this, by utilizing the conservation laws and establishing local boundedness, we derive a blow-up result that precisely determines the blow-up time. Finally, the ill-posedness of the new fifth-order Camassa-Holm type equation in the critical Sobolev space $H^{\frac{1}{2}}$ is established via a norm inflation argument.
	
	
	\vspace*{5pt}
\end{abstract}

\vspace*{10pt}

\tableofcontents

	\section{Introduction}\par
	\quad\quad The Camassa-Holm (CH) equation 
	$$u_t-u_{xxt}+3uu_x=2u_xu_{xx}+uu_{xxx} $$
	is a profoundly significant nonlinear partial differential equation in the fields of fluid dynamics and integrable systems, which has attracted increasing attention in recent years. The CH equation was first introduced in the context of hereditary symmetries by Fuchssteiner and Fokas \cite{Fokas1981} in 1981. It was later explicitly derived as a model for shallow water waves by Camassa and Holm \cite{Camassa1993} in 1993. This equation possesses a rich mathematical structure. Notably, it is completely integrable \cite{Camassa1993,Con2,Con3}, admits a bi-Hamiltonian formulation \cite{Con4,Fokas1981}, and possesses peaked soliton solutions (peakons) of the form $ce^{-|x-ct|}$ \cite{Camassa1993}. These solutions describe traveling waves of largest amplitude \cite{Con6,Con7} and possess orbital stability \cite{Con7}. In addition, the CH equation can characterize wave-breaking phenomenon \cite{Beals2000,Con5,Whitham1999}. Furthermore, local well-posedness of CH equation in Sobolev spaces and Besov spaces has been shown in \cite{Con8,Danchin2001,Danchin2003,Li.J2016}. \par
	
	We note that the nonlinearity in the Camassa–Holm equation is quadratic. Recently, Camassa-Holm type equations with cubic nonlinearity have attracted considerable attention. One particularly renowned Camassa-Holm type equation with cubic nonlinearity is the modified Camassa-Holm (mCH) equation, also known as the Fokas–Olver–Rosenau–Qiao (FORQ) equation \cite{Fokas1995,Fuchssteiner1996,Olver1996,Qiao2006}
	$$m_t+\left[(u^2-u_x^2)m\right]_x=0,\quad m=u-u_{xx}.$$
	It was shown in \cite{Fu2013} that the mCH equation is locally well-posed in the Sobolev space $H^s$ for $s> \frac{5}{2}$. Additionally, the authors of \cite{Fu2013,Gui2013} investigated its local well-posedness in Besov spaces and analyzed its blow-up phenomena. Moreover, it was established in \cite{Gui2013} that the following Hamiltonian functionals are conserved for strong solutions $u(t, x)$ of the mCH equation on $\mathbb{R}$:
	$$ H_1(u)=\int_{\mathbb{R}}\left(u^2+u_x^2\right)dx, \quad H_2(u)=\int_{\mathbb{R}}\left(u^4+2u^2u_x^2-\frac{1}{3}u_x^4\right)dx.$$\par
	In recent years, there has also been growing interest in higher-order Camassa-Holm type equations. A representative example is the fifth-order  Camassa-Holm (FOCH) equation introduced by Liu and Qiao \cite{Liu2018}	
	$$m_t+m_xu+bmu_x=0,\quad m=(1-\alpha^2\partial_x^2)(1-\beta^2\partial_x^2)u,\quad \alpha\neq\beta,\, \alpha\beta\neq0$$
	with any constant $b$ and two arbitrary real constants $\alpha, \beta$. For the case $\alpha = \beta = 1$ and $b=2$, McLachlan and Zhang  \cite{McLachlan2009} established local well-posedness in $H^s$ for $s > \frac{7}{2}$. Subsequently, Tang and Liu \cite{Tang2015} proved this result in the critical Besov space $B^{7/2}_{2,1}$, and  in $B^s_{p,r}$ for $1 \leq p, r \leq +\infty$ with $s > \max\{3 + \frac{1}{p}, \frac{7}{2}\}$, while also demonstrating ill-posedness in $B^{7/2}_{2,\infty}$. For the case $b = 2$ with $m = u - u_{xx} + u_{xxxx}$, local well-posedness was shown to hold in $H^s$ for $s > \frac{9}{2}$ in \cite{Tian2011}. For the general case, the local and global existence of solutions for the FOCH model were established in \cite{Zhu2021}.
	
	 In this paper, we investigate the Cauchy problem for a new fifth-order Camassa-Holm type equation derived by Li \cite{Li.Z}, 
	\begin{align}\label{eq0}
	\begin{split}
		u_{t}-2u_{xxt}&+ u_{xxxxt}+\left( 2u^{3}+uu_{x}^{2}+u^{2}u_{xx}-18u_{x}^{2}u_{xx}+3uu_{xx}^{2}\right.\\
		&+\left. 4uu_{x}u_{xxx}+u^{2}u_{xxxx}+u_{xx}^{3}+4u_{x}u_{xx}u_{xxx}+u_{x}^{2}u_{xxxx}\right)_{x}=0,
	\end{split}
	\end{align}
which is characterized by both cubic nonlinearities and higher-order dispersion. Specifically, drawing inspiration from the mCH equation, Li derived \eqref{eq0}
by leveraging conserved quantities. From Li \cite{Li.Z}, equation \eqref{eq0} admits two fundamental invariants, denoted as 
$$E(u)=\int_{\mathbb{R}}\Big(u^2+2u_x^2+u_{xx}^2\Big)dx$$
and 
$$F(u)=\int_{\mathbb{R}}\Big(u^4-u^2u_x^2+\frac{10}{3}u_x^4+u^2u_{xx}^2+u_x^2u_{xx}^2\Big)dx,$$
which will play a crucial role in the analysis of the blow-up phenomena presented in this work.  Moreover, Li proved that the equation admits pseudo peakons and that these pseudo peakons are orbitally stable with respect to small perturbations in the energy space. Compared with the mCH equation and the FOCH equation, the equation under consideration still has numerous properties worthy of investigation. Our primary focus will be on analyzing the local well-posedness, blow-up phenomena and ill-posedness of this equation in Besov spaces.\par
We now present several equivalent forms of equation \eqref{eq0}, which will facilitate the proof of subsequent theorems. By setting $P(D)=(1-\partial_x^2)^{-2},~P_1(D)=(1-\partial_x^2)^{-1}$, \eqref{eq0} can be  rewritten as
	\begin{equation}\label{eqn}
		\left\{\begin{array}{l}
			n_{t}+(u^{2}+u_{x}^{2})n_{x}=G(u),\\
			n=(1-\partial_x^2)u,\\
			u|_{t=0}=u_{0},
		\end{array}\right.
	\end{equation} 
	where $G(u)=P_1(D)G_1(u)+\partial_xP_1(D)G_2(u)+\partial_x^2P_1(D)G_3(u)$, with $$G_1(u)=2u_xu_{xx}^2,~ G_2(u)=-\frac{5}{3}u^3-2uu_x^2-3u^2u_{xx}+16u_x^2u_{xx}-uu_{xx}^2,~ G_3(u)=-u_xu_{xx}^2-2uu_xu_{xx}.$$
	After some transformations, \eqref{eq0} can also be changed into a transport-like equation
\begin{equation}\label{eq2}
		\left\{\begin{array}{l}
			u_{t}+(u^{2}+\frac{1}{3}u_{x}^{2})u_{x}=P(D)F_1(u)+\partial_xP(D)F_2(u)+\partial_x^2P(D)F_3(u),\\
			u|_{t=0}=u_{0},
		\end{array}\right.
	\end{equation}
with
$F_1(u)=\frac{1}{3}u_x^3,~F_2(u)=-\frac{5}{3}u^3-5uu_x^2-3u^2u_{xx}+24u_x^2u_{xx}-uu_{xx}^2,~F_3(u)=u_xu_{xx}^2+4uu_xu_{xx}.$\par
This paper firstly investigates the local well-posedness for the Cauchy problem \eqref{eqn} (see Theorem \ref{the3.1}). The proof of the local well-posedness is inspired by the argument of approximate solutions by \cite{BCD}. However, in this paper, we apply Moser-type inequalities to obtain refined estimates.
 Having established the local well-posedness of solutions to equation \eqref{eqn}, we first derive a blow-up criterion of the solution to \eqref{eqn} (see Theorem \ref{thm4.2}). Furthermore, under the additional assumption that $u_0\in H^2$, we derive a corollary to this blow-up criterion (see Corollary \ref{cor4.3}).  
 Subsequently, the blow-up analysis for equation \eqref{eqn} becomes particularly challenging due to its higher-order nonlinearities and the lack of a sign-preservation property. Nevertheless, a crucial tool remains at our disposal: the conservation laws. By utilizing the conserved quantity $E(u)$ and establishing the local boundedness of $u_x$ via the Lagrangian coordinate transformation, we establish a blow-up result that characterizes the blow-up time (see Theorem \ref{thm4.4}). For the ill-posedness analysis, we prove the phenomenon of norm inflation by constructing suitable initial data, which implies the ill-posedness of \eqref{eqn} in the Sobolev space $H^{\frac{1}{2}}$ (see Theorem \ref{thm5.1}). \par
The paper is organized as follows. In Section 2, we give some preliminaries which will be used in the sequel. In Section 3, we establish the local well-posedness of \eqref{eqn} in $B^s_{p,r}$ with $\{1\leq p,r\leq \infty, s>\max(\frac{1}{2},\frac{1}{p})\}$ by the Littlewood–Paley theory and transport equations theory. In Section 4, we present the corresponding blow-up criteria for equation \eqref{eqn}. Subsequently, we derive a blow-up result that exactly determines the blow-up time. In Section 5, the ill-posedness of equation \eqref{eqn} in the Sobolev space $H^{\frac{1}{2}}$ is established through a proof of norm inflation.

	\section{\textbf{Preliminaries}}

	In this section, we will present some propositions about the Littlewood-Paley decomposition and the nonhomogeneous Besov spaces.

 \begin{prop}\label{prop}\rm{(The Littlewood-Paley decomposition)} \cite{BCD}
  		Let $\mathcal{C}$ be the annulus $\{\xi\in\mathbb{R}^d:\frac 3 4\leq|\xi|\leq\frac 8 3\}$. There exist radial functions $\chi$ and $\varphi$, valued in the interval $[0,1]$, belonging respecitvely to $\mathcal{D}(B(0,\frac{4}{3}))$ and $\mathcal{D}(\mathcal{C})$, and such that
	$$ \forall\xi\in\mathbb{R}^d,~ \chi(\xi)+\sum_{j\geq 0}\varphi(2^{-j}\xi)=1, $$
  	$$ |j-j'|\geq 2\Rightarrow\mathrm{Supp}\ \varphi(2^{-j}\cdot)\cap \mathrm{Supp}\ \varphi(2^{-j'}\cdot)=\emptyset, $$
  	$$ j\geq 1\Rightarrow\mathrm{Supp}\ \chi \cap \mathrm{Supp}\ \varphi(2^{-j}\cdot)=\emptyset. $$
\end{prop}

\begin{rema}{\rm\cite{BCD}}
	\begin{enumerate}[(1)]
		\item  We fix the functions $\chi$ and $\varphi$ in  Proposition \ref{prop}. Then for all $u\in\mathcal{S}'$, the nonhomogeneous dyadic blocks $\Delta_j$ are defined by
		$$
		\Delta_j u=0,\, ~\text{if}~ j\leq -2,~ \Delta_{-1}u=\chi(\mathcal{D})u=\mathcal{F}^{-1}(\chi \mathcal{F}u),$$
		$$\text{and}~\Delta_ju=\varphi(2^{-j}\mathcal{D})u=\mathcal{F}^{-1}(\varphi(2^{-j}\cdot)\mathcal{F}u), ~\text{if}~ j\geq 0.$$
		The nonhomogeneous low-frequency cut-off operator $S_j$ is defined by
		$$
		S_j u=\sum_{j^\prime\leq j-1}\Delta_{j^\prime}u.$$
		\item  According to Young's inequality, we get
		$$\|\Delta_j u\|_{L^p}, \|S_j u\|_{L^p}\leq C\|u\|_{L^p}, ~\forall~ 1\leq p\leq \infty$$
		where the constant $C$ is independent of $j$ and $p$.
		\item The Littlewood-Paley decomposition is given as follows:
		$$ \rm{Id}=\sum_{j\in\mathbb{Z}}\Delta_j.$$ 
	\end{enumerate}
\end{rema}

	Now, we introduce the definition of nonhomogenous Besov spaces.
\begin{defi}{\rm\cite{BCD}}
	Let $s \in \mathbb{R}$ and $1 \leq p, r \leq \infty$. The nonhomogeneous Besov space $B^s_{p,r}(\mathbb{R}^d)$ consists of all tempered distributions $u$ such that
	$$\|u\|_{B^s_{p,r}}\triangleq\Big\|(2^{js}\|\Delta_ju\|_{L^p})_{j\geq -1} \Big\|_{\ell^r(\mathbb{Z})}<\infty.$$
\end{defi}

	\begin{prop}{\rm\cite{BCD}}\label{basic}
		Let $s,s_1,s_2\in\mathbb{R},\ 1\leq p,p_1,p_2,r,r_1,r_2\leq\infty.$ 
		\begin{enumerate}[(1)]
		\item $B^s_{p,r}$ is a Banach space, and is continuously embedded in $\mathcal{S}'$. 
		\item If $r<\infty$, then $\lim\limits_{j\rightarrow\infty}\|S_j u-u\|_{B^s_{p,r}}=0$. If $p,r<\infty$, then $C_0^{\infty}$ is dense in $B^s_{p,r}$. 
		\item If $p_1\leq p_2$ and $r_1\leq r_2$, then $ B^s_{p_1,r_1}\hookrightarrow B^{s-d(\frac 1 {p_1}-\frac 1 {p_2})}_{p_2,r_2}. $
		If $s_1<s_2$, then the embedding $B^{s_2}_{p,r_2}\hookrightarrow B^{s_1}_{p,r_1}$ is locally compact. 
		\item For any $s>0$, $B^s_{p,r}\cap L^{\infty}$ is an algebra. Furthermore,  $B^s_{p,r}\hookrightarrow L^{\infty} \Leftrightarrow s>\frac d p\ \text{or}\ \{s=\frac d p,\ r=1\}.
			\quad $ 
		\item $B^s_{p,r}$ satisfies the Fatou property, namely, if $(u_n)_{n\in\mathbb{N}}$ is a bounded sequence of $B^s_{p,r}$, then an element $u\in B^s_{p,r}$ and a subsequence $(u_{n_k})_{k\in\mathbb{N}}$ exist such that
			$$ \lim_{k\rightarrow\infty}u_{n_k}=u\ \text{in}\ \mathcal{S}'\quad \text{and}\quad \|u\|_{B^s_{p,r}}\leq C\liminf_{k\rightarrow\infty}\|u_{n_k}\|_{B^s_{p,r}}. $$
		\item Let $m\in\mathbb{R}$ and $f$ be a $S^m$- multiplier, (i.e. f is a smooth function and satisfies that for each multi-index $\alpha$, there exists $ C=C(\alpha)$ such that $|\partial^{\alpha}f(\xi)|\leq C(1+|\xi|)^{m-|\alpha|},\ \forall\xi\in\mathbb{R}^d)$. Then the operator $f(D)=\mathcal{F}^{-1}(f\mathcal{F}\cdot)$ is continuous from $B^s_{p,r}$ to $B^{s-m}_{p,r}$.
		\item Interpolation inequality:
			If $s_1<s_2$ and $\theta \in (0,1)$, then we have
			$$ \|u\|_{B^{\theta s_1+(1-\theta)s_2}_{p,r}}\leq \|u\|_{B^{s_1}_{p,r}}^{\theta}\|u\|_{B^{s_2}_{p,r}}^{1-\theta}. $$
		\item Logarithmic interpolation inequality: Let $\epsilon$ be in $(0, 1)$. A constant $C$ exists such that for any $f\in B^\epsilon_{\infty,\infty}$,
			\begin{equation*}
			\|f\|_{L^\infty}\leq \frac{C}{\epsilon}\|f\|_{B^0_{\infty,\infty}}\Big(1+\log\frac{\|f\|_{B^\epsilon_{\infty,\infty}}}{\|f\|_{B^0_{\infty,\infty}}}\Big).
			\end{equation*}
		\end{enumerate}
	\end{prop}

	Note that $\mathcal{F}^{-1}(\frac{1}{(1+\xi^2)^2})=\frac{1}{4}e^{-|\cdot|}(1+|\cdot |)$ and $\mathcal{F}^{-1}(\frac{1}{1+\xi^2})=\frac{1}{2}e^{-|\cdot|}$. By virtue of Young’s inequality, direct calculations yield the following lemma. 
\begin{lemm} \label{bdd oper}
	Let $1\leq p\leq \infty$,  $P(D)=(1-\partial^2_x)^{-2}$ is an $S^{-4}$-multiplier and $ P_1(D)=(1-\partial^2_x)^{-1}$ is an $S^{-2}$-multiplier. Moreover, the operators $P(D),~\partial_x P(D), ~\partial_x^2P(D)$, $\partial_x^3P(D),~\partial_x^4P(D)$ and $P_1(D),$  $\partial_xP_1(D),~\partial_x^2P_1(D)$  are all bounded in $L^p(\mathbb{R})$.
\end{lemm}

	\begin{prop}{\rm\cite{BCD}}\label{duality}
		Let $s\in\mathbb{R},\ 1\leq p,r\leq\infty.$
		\begin{equation*}\left\{
			\begin{array}{l}
				B^s_{p,r}\times B^{-s}_{p',r'}\longrightarrow\mathbb{R},  \\
				(u,\phi)\longmapsto \sum\limits_{|j-j'|\leq 1}\langle \Delta_j u,\Delta_{j'}\phi\rangle,
			\end{array}\right.
		\end{equation*}
		defines a continuous bilinear functional on $B^s_{p,r}\times B^{-s}_{p',r'}$. Denoted by $Q^{-s}_{p',r'}$ the set of functions $\phi$ in $\mathcal{S}$ such that
		$\|\phi\|_{B^{-s}_{p',r'}}\leq 1$. If $u\in \mathcal{S}'$, then we have
		$$\|u\|_{B^s_{p,r}}\leq C\sup_{\phi\in Q^{-s}_{p',r'}}\langle u,\phi\rangle.$$
	\end{prop}

	Let us now establish some fundamental estimates that will be frequently employed in the proofs of later theorems. 
	\begin{lemm}\label{Moser}{\rm\cite{BCD,Danchin2001,He2017}}
	\begin{enumerate}[(1)]
		\item For any $s>0$ and any $(p,r)$ in $[1,\infty]^2$, there exists a constant $C=C(s,d)$ such that
		$$ \|uv\|_{B^s_{p,r}}\leq C(\|u\|_{L^{\infty}}\|v\|_{B^s_{p,r}}+\|u\|_{B^s_{p,r}}\|v\|_{L^{\infty}}). $$
		\item If $1\leq p,r\leq \infty,\ s_1\leq s_2,\ s_2>\frac{d}{p} (s_2 \geq \frac{d}{p}\ \text{if}\ r=1)$ and $s_1+s_2>\max(0, \frac{2d}{p}-d)$, there exists a constant $C>0$ such that
		$$ \|uv\|_{B^{s_1}_{p,r}}\leq C\|u\|_{B^{s_1}_{p,r}}\|v\|_{B^{s_2}_{p,r}}. $$
		\item If $1\leq p,r\leq \infty,  s>\max(\frac{1}{2},\frac{1}{p})$, there exists a constant $C>0$ such that
		$$\|uv\|_{B^{s-1}_{p,r}(\mathbb{R})}\leq C\|u\|_{B^{s-1}_{p,r}(\mathbb{R})}\|v\|_{B^{s}_{p,r}(\mathbb{R})},$$
		$$ \|uv\|_{B^s_{p,r}(\mathbb{R})}\leq C\|u\|_{B^{s+1}_{p,r}(\mathbb{R})}\|v\|_{B^s_{p,r}(\mathbb{R})}.$$
	\end{enumerate}
\end{lemm}

	Now we state some useful results in the transport equation theory, which are important to the proofs of our  theorems later.
	\begin{equation}\label{transport}
		\left\{\begin{array}{l}
			f_t+v\cdot\nabla f=g,\ x\in\mathbb{R}^d,\ t>0, \\
			f(0,x)=f_0(x).
		\end{array}\right.
	\end{equation}
	\begin{lemm}\label{existence}{\rm\cite{BCD}}
		Let $1\leq p\leq p_1\leq\infty,\ 1\leq r\leq\infty,\ s> -d\min(\frac 1 {p_1}, \frac 1 {p'})$. Let $f_0\in B^s_{p,r}$, $g\in L^1([0,T];B^s_{p,r})$, and let $v$ be a time-dependent vector field such that $v\in L^\rho([0,T];B^{-M}_{\infty,\infty})$ for some $\rho>1$ and $M>0$, and
		$$
		\begin{array}{ll}
			\nabla v\in L^1([0,T];B^{\frac d {p_1}}_{p_1,\infty}\cap L^\infty), &\ \text{if}\ s<1+\frac d {p_1}, \\
			\nabla v\in L^1([0,T];B^{s-1}_{p_1,r}), &\ \text{if}\ s>1+\frac d {p_1}\ \text{or}\ \{s=1+\frac d {p_1}\ \text{and}\ r=1\}.
		\end{array}
		$$
		Then the equation \eqref{transport} has a unique solution $f$ in \\
		-the space $C([0,T];B^s_{p,r})$, if $r<\infty$; \\
		-the space $\Big(\bigcap_{s'<s}C([0,T];B^{s'}_{p,\infty})\Big)\bigcap C_w([0,T];B^s_{p,\infty})$, if $r=\infty$.
	\end{lemm}

\begin{lemm}[A priori estimates in Besov spaces]\label{BCD Thm3.14}{\rm\cite{BCD}}
		Let $s\in\mathbb{R},~1\leq p\leq p_1\leq\infty,~1\leq r\leq \infty$. Assume that
		$$ s> -d\min(\frac{1}{p_1},\frac{1}{p^\prime}).$$
		There exists a constant $C$, such that for all solutions $f\in L^{\infty}([0,T];B^s_{p,r})$ of \eqref{transport} with initial data $f_0\in B^s_{p,r}$,~$g\in L^1([0,T];B^s_{p,r})$ and 
		$$ \begin{array}{ll}
		\nabla v\in L^1([0,T];B^{\frac d {p_1}}_{p_1,\infty}\cap L^\infty), &\ \text{if}\ s<1+\frac d {p_1}, \\
		\nabla v\in L^1([0,T];B^{s-1}_{p_1,r}), &\ \text{if}\ s>1+\frac d {p_1}\ \text{or}\ \{s=1+\frac d {p_1}\ \text{and}\ r=1\},
		\end{array}
		$$
		we have, for a.e. $t\in [0,T],$
		$$ \|f(t)\|_{B^s_{p,r}}\leq e^{CV_{p_1}(t)}\Big(\|f_0\|_{B^s_{p,r}}+\int_0^t e^{-CV_{p_1}(t')}\|g(t')\|_{B^s_{p,r}}{\rm d}t'\Big)$$
		or
		$$ \|f(t)\|_{B^s_{p,r}}\leq \|f_0\|_{B^s_{p,r}}+\int_0^t \|g(t')\|_{B^s_{p,r}}dt'+\int_0^t CV'_{p_1} (t')\|f(t')\|_{B^s_{p,r}}dt{'} $$
		with
		\begin{equation*}
			V_{p_1}'(t)=\left\{\begin{array}{ll}
				
				\|\nabla v(t) \|_{B^{\frac{d}{p_1}}_{p_1,\infty}\cap L^\infty},\ &\text{if}\ s<1+\frac{d}{p_1},\\
				\|\nabla v(t)\|_{B^{s-1}_{p_1,r}},\ &\text{if}\ s>1+\frac{d}{p_1}\ {\rm or}\  \left\{ s=1+\frac{d}{p_1} \ {\rm and}\  r=1\right\}.
			\end{array}\right.
		\end{equation*}	
	\end{lemm}

	For the 1-D linear transport equation, the prior estimates were improved in Reference \cite{Li.J2016} as follows.
	\begin{lemm}\label{Li}{\rm\cite{Li.J2016}}
		Let $s\in\mathbb{R},~1\leq p,r\leq\infty$. Assume that
		$$ s> \max(\frac{1}{2},\frac{1}{p}).$$
		There exists a constant $C$, such that for all solutions $f\in L^{\infty}([0,T];B^{s-1}_{p,r}(\mathbb{R}))$ of \eqref{transport} with initial data $f_0\in B^{s-1}_{p,r}(\mathbb{R})$, $g\in L^1([0,T];B^{s-1}_{p,r}(\mathbb{R}))$ and $\partial_{x} v\in L^1([0,T];B^{s}_{p,r}(\mathbb{R}))$, we have, for a.e. $t\in [0,T],$
		$$ \|f(t)\|_{B^{s-1}_{p,r}}\leq e^{CV(t)}\Big(\|f_0\|_{B^{s-1}_{p,r}}+\int_0^t e^{-CV(t')}\|g(t')\|_{B^{s-1}_{p,r}}{\rm d}t'\Big),$$
		with
		\begin{equation*}
		V(t)=\int_{0}^t\|\partial_{x}v(t')\|_{B^{s}_{p,r}(\mathbb{R})}dt'.
		\end{equation*}	
	\end{lemm}

	Let us consider the following initial value problem
	\begin{equation}\label{eqflow}
		\left\{\begin{array}{l}
			y_t(t,x)=(u^2+u_x^2)(t,y),~t\in [0,T),  \\
			y(0,x)=x,~x\in\mathbb{R}.
		\end{array}\right.
	\end{equation}
	By applying classical results in the theory of ordinary differential equations, we have the following lemma.
	\begin{lemm}{\rm\cite{Cons1}}
		 Let u $\in C( [ 0,T);H^s(\mathbb{R}))\cap  C^1([0,T);H^{s-1}(\mathbb{R})),~s\geq 2$ .Then \eqref{eqflow} has a unique solution $ y \in C^1([0,T)\times \mathbb{R};\mathbb{R})$. Moreover, the map $y(t,\cdot)$ is an increasing diffeomorphism of $\mathbb{R}$ with  
		$$y_x(t,x)=\exp\Big(\int_{0}^{t} (u^2+u_x^2)_x(s,y(s,x))ds\Big)>0,~\forall(t,x)\in [0,T)\times \mathbb{R}.$$ 
	\end{lemm}

   \begin{lemm}\label{continue}{\rm\cite{Li.J}}
   	Suppose that $1\leq p\leq\infty, 1\leq r<\infty,  s>\frac 1 p$ (or $s=\frac 1 p, 1<p<\infty, r=1)$. Denote $\bar{\mathbb{N}}=\mathbb{N}\cup\{\infty\}$. Let $(v^n)_{n\in\bar{\mathbb{N}}}\subset C([0,T];B^{s}_{p,r})$, and $(a^n)_{n\in\bar{\mathbb{N}}} \subset C([0,T];B^{s+1}_{p,r})$. Assume that $v^n$ solves 
   	\begin{equation}
   		\left\{\begin{array}{l}
   			v^n_t+a^n\cdot\nabla v^n=f, \\
   			v^n(0,x)=v_0(x),
   		\end{array}\right.
   	\end{equation}
   	with $v_0\in B^s_{p,r},\ f\in L^1([0,T];B^s_{p,r})$, and for some $\alpha\in L^1([0,T])$ which satisfies
   	$$\sup\limits_{n\in\bar{\mathbb{N}}}\|a^n(t)\|_{B^{s+1}_{p,r}}\leq \alpha(t).$$
   	If $a^n \rightarrow a^{\infty}$ in $L^1([0,T];B^s_{p,r})$ as $n\rightarrow \infty$, then $v^n \rightarrow v^{\infty}$ in ~$C([0,T];B^s_{p,r})$ .
   \end{lemm}

\section{Local well-posedness}
In this section, we obtain the local well-posedness of equation \eqref{eqn} with data in Besov spaces. Before starting our local existence result, we introduce the following function spaces:
\begin{flalign}
&E_{p,r}^{s}(T)\triangleq\mathcal{C}([0,T];B_{p,r}^{s})\cap\mathcal{C}^{1}([0,T];B_{p,r}^{s-1}),\quad\mathrm{if}~r<\infty, \nonumber \\
&E_{p,\infty}^{s}(T)\triangleq\mathcal{C}_{w}([0,T];B_{p,\infty}^{s})\cap C^{0,1}
([0,T];B_{p,\infty}^{s-1}),
\nonumber
\end{flalign}
with $T>0,~s\in\mathbb{R}$ and $1\leq p,r\leq\infty$.

Our main theorem can be stated as follows.

\begin{theo}\label{the3.1}
	Let $1\leq p,r\leq\infty,~s>\max(\frac{1}{2},\frac{1}{p})$ and $n_0\triangleq (1-\partial_x^2) u_0$  be in $B^{s}_{p,r} $. Then there exists a time $T>0$ such that \eqref{eqn} has a unique solution $n$ in $E^s_{p,r}(T)$. Moreover,  the solution $n$ depends continuously on the initial data $n_0$.

\end{theo}
\begin{proof}\
	Now  we prove Theorem \ref{the3.1} in five steps.\\

	\textbf{Step 1: Constructing Approximate Solutions}
	
	Starting from $u^0\triangleq 0$, we define by induction a sequence $(n^k)_{k\in\mathbb{N}}$
	of smooth functions by solving the following linear transport equation:
	\begin{equation}\label{T_k}
		(T_k)\left\{\begin{array}{l}
			n_t^{k+1}+ ((u^k)^2+(u_x^k)^2)n^{k+1}_x=G(u^k), \\
			u^{k+1}=(1-\partial^2_x)^{-1}n^{k+1},  \\
			u^{k+1}|_{t=0}=S_{k+1}u_0.
		\end{array}\right.
	\end{equation}
By induction, we assume that $n^k\in E^s_{p,r}(T)$ for all positive $T$. For the right side of \eqref{T_k}, the following estimates can be derived by using Lemma \ref{bdd oper},  $B^{s}_{p,r}\hookrightarrow L^\infty$ and Lemma \ref{Moser}.
\begin{align}\label{G_est1}
	\|G(u^k)\|_{B^{s}_{p,r}}&=\|P_1(D)G_1(u^k)+\partial_xP_1(D)G_2(u^k)+\partial_x^2P_1(D)G_3(u^k)\|_{B^{s}_{p,r}}\notag\\
&\leq
C\Big(\|2u_x^k(u_{xx}^k)^2\|_{B^{s}_{p,r}}+\|-\frac{5}{3}(u^k)^3-2u^k(u_x^k)^2-3(u^k)^2u^k_{xx}+16(u_x^k)^2u^k_{xx}-u^k(u_{xx}^k)^2\|_{B^{s}_{p,r}}\notag\\
&~~~+\|-u^k_x(u^k_{xx})^2-2u^ku^k_xu^k_{xx}\|_{B^{s}_{p,r}}\Big) \notag \\ 
&\leq C\|n^k\|^3_{B^{s}_{p,r}}.
\end{align}
To apply Lemma \ref{existence}, we first verify that its hypotheses hold. 
$$\|n_0^{k+1}\|_{B^{s}_{p,r}}=\|S_{k+1}n_0\|_{B^{s}_{p,r}}\leq C\|n_0\|_{B^{s}_{p,r}}.$$
We then deduce that $n_0^{k+1}\in B^{s}_{p,r}.$
From the previous computation, we have $G(u^k)\in L^1([0,T];B^{s}_{p,r})$. At the same time, applying Bernstein's inequalities in \cite{BCD} and Lemma  \ref{Moser} followed by a simple calculation shows that $(u^k)^2+(u_x^k)^2\in L^\rho([0,T];B^{-M}_{\infty,\infty})$ for some $\rho>1$ and $M>0$. 
By virtue of Lemma \ref{Moser} and the fact that $B^{s}_{p,r}$ is an algebra, we then have
$$ \int_0^t \|\nabla((u^k)^2+(u_x^k)^2)\|_{B^{s}_{p,r}} d\tau =\int_0^t \|2u^ku_x^k+2u_x^ku_{xx}^k\|_{B^{s}_{p,r}} d\tau \leq C\int^t_0\|n^k\|^2_{B^{s}_{p,r}}d\tau.$$
Since $B^{s}_{p,r} \hookrightarrow B^{1/p}_{p,\infty}\cap L^\infty,~ B^{s}_{p,r}\hookrightarrow B^{s-1}_{p,r}$ and the previous process, we can use Lemma \ref{existence}. Making use of Lemma \ref{existence} ensures that \eqref{T_k} has a unique solution $n^{k+1}$ in\\
-the space $C([0,T];B^s_{p,r})$,~if $r<\infty$; \\
-the space $\Big(\bigcap_{s'<s}C([0,T];B^{s'}_{p,\infty})\Big)\bigcap C_w([0,T];B^s_{p,\infty})$,~if $r=\infty$.\par
In the case where $r$ is finite, we still have to prove that $n_t^{k+1} \in C([0,T];B^{s-1}_{p,r}).$ We easily deduce that $G(u^k)\in C([0,T];B^{s-1}_{p,r})$ and $((u^k)^2+(u_x^k)^2)n^{k+1}_x \in  C([0,T];B^{s-1}_{p,r})$ by Lemma \ref{Moser}. From the equation \eqref{T_k}, we have $n_t^{k+1} \in C([0,T];B^{s-1}_{p,r}).$ When $r=\infty$, it suffices to prove $n_t^{k+1} \in L^\infty([0,T];B^{s-1}_{p,\infty}).$ The proof for the case $r=\infty$ is similar to the case $r<\infty$, so we omit the details here. We then can conclude that $n^{k+1}\in E^s_{p,r}(T)$ for all positive $T$.\\

\textbf{Step 2: Uniform Bounds}\par
We define $N^k(t) \triangleq \int^t_0\|n^k(\tau)\|^2_{B^{s}_{p,r}}d\tau$. According to \textbf{Step 1} and Lemma \ref{BCD Thm3.14}, for all $k\in\mathbb{N}$, we have the following inequality,
\begin{align}
\|n^{k+1}(t)\|_{B^{s}_{p,r}} &\leq e^{CN^k(t)}
\Big(\|n^{k+1}_0\|_{B^{s}_{p,r}}+\int_0^t e^{-CN^k(\tau)} \|G^k(\tau)\|_{B^{s}_{p,r}}d\tau\Big)  \notag\\
&\leq Ce^{CN^k(t)}\Big(\|n_0\|_{B^{s}_{p,r}}+\int_0^t e^{-CN^k(\tau)} \|n^{k}(\tau)\|_{B^{s}_{p,r}}^3d\tau\Big).       \label{1}
\end{align}
Without loss of generality, we may assume $C > 1$. We fix a $T>0$ such that $4C^3T\|n_0\|_{B^{s}_{p,r}}^2 < 1$ and suppose that
\begin{equation}\label{2}
\|n^k(t)\|_{B_{p,r}^s}\leq\frac{C\|n_0\|_{B_{p,r}^s}}{\sqrt{1-4C^3t\|n_0\|_{B_{p,r}^s}^2}},\quad \forall t\in[0,T]. 
\end{equation}
We thus get
\begin{equation}\label{3}
e^{CN^k(t)} \leq e^{C^3\int_0^t \|n_0\|_{B_{p,r}^s}^2 (1-4C^3\tau\|n_0\|_{B_{p,r}^s}^2)^{-1} d\tau} =(1-4C^3t\|n_0\|_{B_{p,r}^s}^2)^{-\frac{1}{4}},
\end{equation}
and
\begin{equation}\label{4}
e^{CN^k(t)-CN^k(\tau)} \leq (1-4C^3t\|n_0\|_{B_{p,r}^s}^2)^{-\frac{1}{4}}(1-4C^3\tau\|n_0\|_{B_{p,r}^s}^2)^{\frac{1}{4}}.
\end{equation}
Plugging \eqref{2}, \eqref{3} and \eqref{4} into \eqref{1} yields
\begin{align*}
\|n^{k+1}(t)\|_{B^{s}_{p,r}} 
&\leq C(1-4C^3t\|n_0\|_{B_{p,r}^s}^2)^{-\frac{1}{4}} \Big(\|n_0\|_{B^{s}_{p,r}}+\int_0^t \frac{C^3\|n_0\|_{B_{p,r}^s}^3}{(1-4C^3\tau\|n_0\|_{B_{p,r}^s}^2)^{\frac{5}{4}}} d\tau\Big)  \notag\\
&\leq C(1-4C^3t\|n_0\|_{B_{p,r}^s}^2)^{-\frac{1}{2}} \|n_0\|_{B_{p,r}^s}.     
\end{align*}                                                                                                                                                           
By induction, the sequence $(n^k)_{k\in\mathbb{N}}$ is bounded in $L^\infty([0,T];B^{s}_{p,r})$. This clearly entails that $(G(u^k))_{k\in\mathbb{N}}$ and $(((u^k)^2+(u_x^k)^2)n^{k}_x)_{k\in\mathbb{N}}$ are bounded in $L^\infty([0,T];B^{s-1}_{p,r})$ by the first step. According to the equation \eqref{T_k}, we can conclude that $(n^k)_{k\in\mathbb{N}}$ is bounded in $E^s_{p,r}(T)$.\\

\textbf{Step 3: Convergence and Regularity}\par

We claim that $(n^k)_{k\in\mathbb{N}}$ is a Cauchy sequence in $C([0,T];B^{s-1}_{p,r})$. Indeed, for all $k, p\in\mathbb{N}$, denote $w^{k,p}=n^{k+p}-n^k$. By the system $(T_{k+p})$ and $(T_k)$, $w^{k+1,p}$ satisfies	
\begin{align}\label{W_{k,p}}
\left\{\begin{array}{l}
w_t^{k+1,p}+ ((u^{k+p})^2+(u_x^{k+p})^2)w^{k+1,p}_x=G(u^{k+p})-G(u^{k})-(u^{k}+u^{k+p})(u^{k+p}-u^{k})n^{k+1}_x\\
~~~~~~~~~~~~~~~~~~~~~~~~~~~~~~~~~~~~~~~~~~~~~~\quad~ -(u^{k}_x+u^{k+p}_x)(u^{k+p}_x-u^{k}_x)n^{k+1}_x,\\
w^{k+1,p}|_{t=0}=S_{k+p+1}n_0-S_{k+1}n_0,
\end{array}\right.
\end{align}	
where 
\begin{align*}
G(u^{k+p})-G(u^{k})&=P_1(D)(G_1(u^{k+p})-G_1(u^k))+\partial_xP_1(D)(G_2(u^{k+p})-G_2(u^k))\\
&~~~+\partial_x^2P_1(D)(G_3(u^{k+p})-G_3(u^k)).
\end{align*}	
We now estimate the right-hand side of \eqref{W_{k,p}}. Denote 
\begin{align*}
\Rmnum{1}&\triangleq\|G(u^{k+p})-G(u^{k})-(u^{k}+u^{k+p})(u^{k+p}-u^{k})n^{k+1}_x-(u^{k}_x+u^{k+p}_x)(u^{k+p}_x-u^{k}_x)n^{k+1}_x \|_{B^{s-1}_{p,r}} \\
&\leq \Rmnum{1}_1+ \Rmnum{1}_2,
\end{align*} 
where
\begin{align*}
\begin{array}{l}
\Rmnum{1}_1\triangleq \|G(u^{k+p})-G(u^{k})\|_{B^{s-1}_{p,r}},\\
\Rmnum{1}_2\triangleq \|-(u^{k}+u^{k+p})(u^{k+p}-u^{k})n^{k+1}_x-(u^{k}_x+u^{k+p}_x)(u^{k+p}_x-u^{k}_x)n^{k+1}_x \|_{B^{s-1}_{p,r}}.
\end{array}
\end{align*}

\textit{Bounds for $\Rmnum{1}_1$.} By virtue of Lemma \ref{bdd oper} and Lemma \ref{Moser}, using the fact that $(n^k)_{k\in\mathbb{N}}$ has uniform bounds in $
L^\infty([0,T];B^{s}_{p,r})$, we get 
\begin{align}\label{I_1}
\Rmnum{1}_1&\leq
C\Big(\|G_1(u^{k+p})-G_1(u^k)\|_{B^{s-1}_{p,r}}+\|G_2(u^{k+p})-G_2(u^k)\|_{B^{s-1}_{p,r}}+\|G_3(u^{k+p})-G_3(u^k)\|_{B^{s-1}_{p,r}}\Big) \notag\\
&\leq
C\Big(\|2u_x^{k+p}(u_{xx}^{k+p})^2-2u_x^k(u_{xx}^k)^2\|_{B^{s-1}_{p,r}}+\|(-\frac{5}{3}(u^{k+p})^3-2u^{k+p}(u_x^{k+p})^2-3(u^{k+p})^2u^{k+p}_{xx}+16(u_x^{k+p})^2\notag\\
&~~~~~~~\times u^{k+p}_{xx}-u^{k+p}(u_{xx}^{k+p})^2)-(-\frac{5}{3}(u^k)^3-2u^k(u_x^k)^2-3(u^k)^2u^k_{xx}+16(u_x^k)^2u^k_{xx}-u^k(u_{xx}^k)^2)\|_{B^{s-1}_{p,r}}\notag\\
&~~~~~~~+\|u^{k+p}_x(u^{k+p}_{xx})^2+2u^{k+p}u^{k+p}_xu^{k+p}_{xx}-u^k_x(u^k_{xx})^2-2u^ku^k_xu^k_{xx}\|_{B^{s-1}_{p,r}}\Big) \notag \\ 
&\leq C\|n^{k+p}-n^k\|_{B^{s-1}_{p,r}}.
\end{align}

\textit{Bounds for $\Rmnum{1}_2$.}
Similar arguments lead to
\begin{align}\label{I_2}
\Rmnum{1}_2 &\leq \|u^{k}+u^{k+p}\|_{B^{s}_{p,r}}\|u^{k+p}-u^{k}\|_{B^{s}_{p,r}}\|n_x^{k+1}\|_{B^{s-1}_{p,r}}+\|u^{k}_x+u^{k+p}_x\|_{B^{s}_{p,r}}\|u^{k+p}_x-u^{k}_x\|_{B^{s}_{p,r}}\|n_x^{k+1} \|_{B^{s-1}_{p,r}}\notag \\
&\leq C\|n^{k+p}-n^k\|_{B^{s-1}_{p,r}}.
\end{align}
From \eqref{I_1} and \eqref{I_2}, the right-hand side of the system \eqref{W_{k,p}} admits the following estimate:
$$\Rmnum{1}\leq C\|n^{k+p}-n^k\|_{B^{s-1}_{p,r}}.$$	
Next, if $r<\infty$, we derive the estimate for the initial data in \eqref{W_{k,p}}:
\begin{align}
\|S_{k+p+1}n_0-S_{k+1}n_0\|_{B^{s-1}_{p,r}}&=\|\sum_{j'=k+1}^{k+p}\Delta_{j'}n_0\|_{B^{s-1}_{p,r}}\notag\\
&\leq\Big( \sum_{l\geq -1}2^{l(s-1)r} \|\Delta_l \sum_{j'=k+1}^{k+p}\Delta_{j'}n_0\|_{L^{p}}^r \Big)^{1/r} \notag\\
&\leq C\Big( \sum_{l=k}^{k+p+1}2^{-lr}2^{lsr}\|\Delta_{l}n_0\|_{L^{p}}^r \Big)^{1/r}\notag\\
&\leq C \|n_0\|_{B^{s}_{p,\infty}}\Big( \sum_{l=k}^{k+p+1}2^{-lr}\Big)^{1/r} \leq C 2^{-k}.
\end{align}
The corresponding results for $r = \infty$ require only trivial modifications to the argument.	By virtue of Lemma \ref{Li} and using the fact that  $(m^k)_{k\in\mathbb{N}}$ is bounded in $
L^\infty([0,T];B^{s}_{p,r})$,
we infer that
\begin{align}
\|w^{k+1,p}\|_{B^{s-1}_{p,r}} &\leq C_T\Big(\|S_{k+p+1}n_0-S_{k+1}n_0\|_{B^{s-1}_{p,r}}+\int_{0}^{t}\|n^{k+p}-n^k\|_{B^{s-1}_{p,r}} dt' \Big) \notag\\
&\leq C_T\Big(2^{-k}+\int_{0}^{t}\|w^{k,p}\|_{B^{s-1}_{p,r}} dt' \Big).
\end{align}
Arguing by induction, for any $t\in [0,T],$ we get, 	
\begin{align}
\|w^{k+1,p}(t)\|_{B^{s-1}_{p,r}} &\leq C_T\Big(2^{-k}+\int_{0}^{t}C_T(2^{-(k-1)}+\int_{0}^{t_1}\|w^{k-1,p}\|_{B^{s-1}_{p,r}} dt_2) dt_1 \Big) \notag\\
&\leq C_T2^{-k}+C_T^22^{-k+1}t+\cdots+C_T^{k+1}\frac{1}{k!}2^{-k+k}t^k\notag\\
 &~~~+C_T^{k+1}\int_{0}^{t}\cdots\int_{0}^{t_k}\|w^{0,p}\|_{B^{s-1}_{p,r}}dt_{k+1}\cdots dt_1 \notag\\
&\leq C_T\sum_{l=0}^{k}\frac{(2C_T T)^l}{l!}2^{-k}+C_T^{k+1}\int_{0}^{t}\cdots\int_{0}^{t_k}\|n^p\|_{B^{s-1}_{p,r}}dt_{k+1}\cdots dt_1 \notag\\
&\leq C\Big( \sum_{l=0}^{k}\frac{(2C_T T)^l}{l!}2^{-k}+ \frac{1}{(k+1)!}(C_T T)^{k+1}\Big) \notag
\end{align}
which leads to $\|w^{k+1,p}\|_{B^{s-1}_{p,r}} \rightarrow 0 $ as $k\rightarrow \infty$ for all $p\in \mathbb{N}$. Hence, $(n^k)_{k\in\mathbb{N}}$ is a Cauchy sequence in $C([0,T];B^{s-1}_{p,r})$ and converges to some limit function $n\in C([0,T];B^{s-1}_{p,r})$.
Since $(n^k)_{k\in\mathbb{N}}$ is bounded in $
L^\infty([0,T];B^{s}_{p,r})$, the Fatou property for Besov spaces guarantees that $n\in L^\infty([0,T];B^{s}_{p,r})$. Passing to the limit in \eqref{T_k}, we easily deduce that $n$ is a solution of \eqref{eqn}.  Thanks to Lemma \ref{existence} and the equation \eqref{eqn}, we conclude that $n \in E^s_{p,r}(T)$.	\\

\textbf{Step 4: Uniqueness.}\par 
	Let $n_1$ and $n_2$ be two solutions in $E^s_{p,r}(T)$ of \eqref{eqn} with initial data $n_{10}, n_{20} \in B^{s}_{p,r}$. In order to show that these two solutions coincide, we shall give estimates for $w\triangleq n_1-n_2$. Straightforward calculations shows that $w$ satisfies the following system: 
	\begin{equation}\label{W12}
	\left\{\begin{array}{l}
	w_t+ (u_1^2+u_{1x}^2)w_x =-(u_1^2+u_{1x}^2)n_{2x}+(u_2^2+u_{2x}^2)n_{2x}+G(u_1)-G(u_2), \\
	u_1=(1-\partial^2_x)^{-1}n_1,\quad u_2=(1-\partial^2_x)^{-1}n_2,  \\
	w|_{t=0}=n_{10}-n_{20}.
	\end{array}\right.
	\end{equation}
	By repeating the computational procedure from \textbf{Step 3}, we derive
	$$\|w(t)\|_{B^{s-1}_{p,r}} \leq C_T\Big(\|n_{10}-n_{20}\|_{B^{s-1}_{p,r}}+\int_{0}^{t}\|w(t')\|_{B^{s-1}_{p,r}} dt' \Big). $$
	Thanks to Gronwall’s inequality, we conclude that
	\begin{equation}\label{unique}
	\sup_{t\in[0,T]}\|w(t)\|_{B^{s-1}_{p,r}}\leq C\|n_{10}-n_{20}\|_{B^{s-1}_{p,r}},
	\end{equation}
	which completes the proof of uniqueness.\\
	
\textbf{Step 5: Continuous Dependence.}	\par
	Denote $\bar{\mathbb{N}}=\mathbb{N}\cup \{\infty\}$. First, we consider the case where $r<\infty $. For $k\in \bar{\mathbb{N}}$, let us suppose $n^k \in C([0,T];B^{s}_{p,r}) $ is the solution of \eqref{eqn} with initial data $n^k_{0} \in B^{s}_{p,r} $ where $k\in\bar{\mathbb{N}}.$ Now we claim that if $n_0^k\rightarrow n_0^\infty$ in $B^{s}_{p,r}$ as $k\rightarrow \infty$, then $n^k\rightarrow n^\infty$ in $C([0,T];B^{s}_{p,r})$ with the lifespan $T$ satisfying $4C^3T\sup_{k\in\bar{\mathbb{N}}}\|n_0^k\|_{B^{s}_{p,r}}^2 < 1$.\par
	Indeed, split $n^k$ into $y^k+z^k$ with $(y^k,z^k)$ satisfying 
	\begin{equation}\label{Ty}
	\left\{\begin{array}{l}
	y^k_{t}+((u^k)^{2}+(u^k_{x})^{2})y^k_{x}=G(u^\infty),\\
	y^k|_{t=0}=n^\infty_{0},
	\end{array}\right.\quad\mathrm{and}\quad
	\left\{\begin{array}{l}
	
	z^k_{t}+((u^k)^{2}+(u^k_{x})^{2})z^k_{x}=G(u^k)-G(u^\infty),\\
	z^k|_{t=0}=n^k_0-n^\infty_{0}.
	\end{array}\right.
	\end{equation}
    Let
    $$H=\frac{C\sup_{k\in\bar{\mathbb{N}}}\|n_0^k\|_{B^{s}_{p,r}}}{\sqrt{1-4C^3T(\sup_{k\in\bar{\mathbb{N}}}\|n_0^k\|_{B^{s}_{p,r}})^2 }}.$$
    According to the proof of the existence, we have 
    $$\|n^k(t)\|_{B^{s}_{p,r}}\leq H,\quad \forall t\in[0,T],~ \forall k\in\bar{\mathbb{N}}. $$
	Hence, $(n^k)_{k\in\bar{\mathbb{N}}}$ is uniformly bounded in $C([0,T];B^{s}_{p,r})$ and
	$$\sup_{k\in\bar{\mathbb{N}}}\|(u^k)^{2}+(u^k_{x})^{2}\|_{B^{s+1}_{p,r}}\leq C \sup_{k\in\bar{\mathbb{N}}}\|n^k\|^2_{B^{s}_{p,r}}\leq CH^2.$$
	Moreover, applying \eqref{unique} together with $n_0^k\rightarrow n_0^\infty$ in $B^{s}_{p,r}$ as $k\rightarrow \infty$ implies that 
	\begin{equation}
	n^k\rightarrow n^\infty ~\text{in}~  C([0,T];B^{s-1}_{p,r})~ \text{as}~ k\rightarrow \infty.  
	\end{equation}
	Thus, owing to the uniform bounds of $(n^k)_{k\in\bar{\mathbb{N}}}$,  we have 
	\begin{align*}
	\|(u^k)^{2}+(u^k_{x})^{2}-(u^\infty)^2-(u_x^\infty)^2\|_{B^{s}_{p,r}}&\leq \|u^k-u^\infty\|_{B^{s}_{p,r}}\|u^k+u^\infty\|_{B^{s}_{p,r}}+\|u^k_x-u_x^\infty\|_{B^{s}_{p,r}}\|u^k_x+u_x^\infty\|_{B^{s}_{p,r}}\\
	&\leq CH\|n^k-n^\infty\|_{B^{s-1}_{p,r}},
	\end{align*}
	which implies that $(u^k)^{2}+(u^k_{x})^{2}\rightarrow (u^\infty)^2+(u_x^\infty)^2$ in $L^1([0,T];B^{s}_{p,r})$.
	Therefore, applying Lemma \ref{continue}, we have $y^k\rightarrow y^\infty$ in $C([0,T];B^{s}_{p,r})$ as $k\rightarrow\infty$.  
	To control $z^k$, we need to estimate $G(u^k)-G(u^\infty)$. Similar to the estimates for $\Rmnum{1}_1$ in \textbf{Step 3} and the application of Lemma \ref{Moser}, we derive 
	$$\|G(u^{k})-G(u^{\infty})\|_{B^{s}_{p,r}}\leq CH^2\|n^{k}-n^{\infty}\|_{B^{s}_{p,r}}. $$
	Applying Lemma \ref{existence} and Lemma \ref{BCD Thm3.14}, we obtain 
	\begin{equation}
	\|z^k(t)\|_{B^{s}_{p,r}}\leq Ce^{CH^2 t}\Big(\|n_0^k-n_0^\infty\|_{B^{s}_{p,r}}+\int_{0}^{t}\|(n^{k}-n^\infty)(t')\|_{B^{s}_{p,r}} dt' \Big).	
	\end{equation}
	Since $z^\infty=0$ and $y^k\rightarrow y^\infty$ in $C([0,T];B^{s}_{p,r})$ as $k\rightarrow\infty$, for any $\epsilon>0$, there exists $\bar{N_1}\in\mathbb{N}$ such that
	$$\|(y^k-n^\infty)(t)\|_{B^{s}_{p,r}}\leq \epsilon,\quad \forall t\in[0,T]$$
	holds for all $k>\bar{N_1}$.
	Therefore, we have
	\begin{align*}
	\|(n^k-n^\infty)(t)\|_{B^{s}_{p,r}}&\leq \|z^k(t)\|_{B^{s}_{p,r}}+\|(y^k-n^\infty)(t)\|_{B^{s}_{p,r}}\\
	&\leq Ce^{CH^2 t}\Big(\|n_0^k-n_0^\infty\|_{B^{s}_{p,r}}+\int_{0}^{t}\|(n^{k}-n^\infty)(t')\|_{B^{s}_{p,r}} dt'\Big)+\epsilon.
	\end{align*}
	The Gronwall inequality implies that
	\begin{equation}
	\|(n^k-n^\infty)(t)\|_{B^{s}_{p,r}}\leq C (\epsilon+\|n_0^k-n_0^\infty\|_{B^{s}_{p,r}} ),\quad  \forall t\in[0,T].
	\end{equation}
	Hence, we gain the continuity of \eqref{eqn} in $C([0,T];B^{s}_{p,r})$ with respect to the initial data in $B^{s}_{p,r}$ for $r<\infty$.\par
	In the case $r=\infty$, we note that according to \eqref{unique} if $n_0^k\rightarrow n_0^\infty$ in $B^{s}_{p,r}$ as $k\rightarrow \infty$, we obtain $n^k\rightarrow n^\infty $ in $L^\infty([0,T];B^{s-1}_{p,r})$. Then, for fixed $\varphi\in B^{-s}_{p',1}$ we write
	
	\begin{align}\label{r=infty}
	\langle n^k-n^\infty,\varphi\rangle &=\langle S_j(n^k-n^\infty),\varphi\rangle +\langle (Id-S_j)(n^k-n^\infty),\varphi\rangle \notag\\
	&=\langle n^k-n^\infty,S_j\varphi\rangle +\langle n^k-n^\infty,(Id-S_j)\varphi\rangle.
	\end{align} 
	By duality (see Proposition \ref{duality}), we have 
	\begin{equation}\label{1r=infty}
	|\langle n^k-n^\infty,(Id-S_j)\varphi\rangle|\leq C\|n^k-n^\infty\|_{B^{s}_{p,\infty}}\|\varphi-S_j\varphi\|_{B^{-s}_{p',1}}\leq C\|\varphi-S_j\varphi\|_{B^{-s}_{p',1}}.
	\end{equation}
	For all $\epsilon>0$, we now choose a sufficiently large $j$ such that
	$$\|\varphi-S_j\varphi\|_{B^{-s}_{p',1}}<\frac{\epsilon}{2C}. $$
	Next, for the first term on the right-hand side of equation \eqref{r=infty}, we have 
	\begin{align}\label{2r=infty}
	|\langle n^k-n^\infty,S_j\varphi\rangle| &\leq C\|n^k-n^\infty\|_{B^{s-1}_{p,\infty}}\|S_j\varphi\|_{B^{1-s}_{p',1}} \notag\\
	&\leq C\|n^k-n^\infty\|_{B^{s-1}_{p,\infty}} \sum_{l\geq -1} 2^{l(1-s)}\|\Delta_lS_j\varphi\|_{L^{p'}}\notag\\
	&\leq C\|n^k-n^\infty\|_{B^{s-1}_{p,\infty}} \sum_{l\geq -1,l\leq j} 2^{l(1-s)}\|\Delta_l\varphi\|_{L^{p'}} \notag\\
	&\leq C2^j\|n^k-n^\infty\|_{B^{s-1}_{p,\infty}} \|\varphi\|_{B^{-s}_{p',1}}.
	\end{align}
	Since $n^k\rightarrow n^\infty $ in $L^\infty([0,T];B^{s-1}_{p,\infty})$, there exists $\bar{N}=\bar{N}(\epsilon)>0$ such that 
	$$\|n^k-n^\infty\|_{B^{s-1}_{p,\infty}}<\frac{\epsilon}{C2^{j+1}\|\varphi\|_{B^{-s}_{p',1}}} $$
	holds for all $k>\bar{N}$.  
	Plugging \eqref{1r=infty} and \eqref{2r=infty} into\eqref{r=infty}, we thus get
	$$\langle n^k-n^\infty,\varphi\rangle\rightarrow 0 ~\text{as}~ k\rightarrow \infty.$$
	This completes the proof of continuous dependence in the case $r=\infty$.\par
	Combining  \textbf{Step 1} to \textbf{Step 5}, we complete  the proof of Theorem \ref{the3.1}.
\end{proof}

\section{Blow up}
\par
This section is devoted to presenting the blow-up statements for equation \eqref{eqn}. First, we present a lemma concerning commutator estimates, which will be used later in establishing the blow-up criterion. 

\begin{lemm}\label{lem4.1}
	Let $s>0$ and $1\leq p,r\leq \infty$. Denote $R_j=[v\cdot\partial_{x},\Delta_j]f.$ There exists a contant $C=C(p,r,s)$ such that 
	\begin{equation}
	\|(2^{js}\|R_j\|_{L^p(\mathbb{R})})_{j\geq -1}\|_{\ell^r}\leq C\Big(\|\partial_x v\|_{L^\infty(\mathbb{R})}\|f\|_{B^s_{p,r}(\mathbb{R})}+\|\partial_x v\|_{B^s_{p,r}(\mathbb{R})}\|f\|_{L^\infty(\mathbb{R})}\Big).
	\end{equation}
\end{lemm}
\begin{proof}
	By splitting $v=S_0 v+\widetilde{v}$ and applying Bony’s decomposition, we end up with the expression $R_j =\sum_{i=1}^{8}R_j^i$, where
	\begin{align*}
	&R_{j}^{1}=[T_{\widetilde{v}},\Delta_{j}]\partial_{x}f,~R_{j}^{2}=T_{\partial_{x}\Delta_{j}f}\widetilde{v},~R_{j}^{3}=-\Delta_{j}T_{\partial_{x}f}\widetilde{v},~R_{j}^{4}=\partial_{x}R(\widetilde{v},\Delta_{j}f),\\ &R_{j}^{5}=-R(\partial_{x}\widetilde{v},\Delta_{j}f),~R_{j}^{6}=-\partial_{x}\Delta_{j}R(\widetilde{v},f),~R_{j}^{7}=\Delta_{j}R(\partial_{x}\widetilde{v},f),~R_{j}^{8}=[S_{0}v,\Delta_{j}]\partial_{x}f.
	\end{align*}
	 In the following computations, we denote by $(c_j)_{j\geq -1}$ a sequence such that $\|(c_j)\|_{\ell^r}\leq 1.$ According to the commutator estimates of \cite{BCD}, we have for $s>0$,
	\begin{equation}\label{Ri}
		2^{js}\|R_j^i\|_{L^p}\leq Cc_j\|\partial_x v\|_{L^\infty}\|f\|_{B^s_{p,r}}\quad \text{for}~i\neq 3.
	\end{equation}
	Now, for $R_j^3=-\Delta_j T_{\partial_{x}f}\widetilde{v},$ we have
	$$R_j^3=-\sum_{|j-j'|\leq 4}\Delta_j(S_{j'-1}\partial_{x} f\Delta_{j'}\widetilde{v})=-\sum_{\substack{|j-j'|\leq 4\\j''\leq j'-2}}\Delta_j(\Delta_{j''}\partial_{x} f\Delta_{j'}\widetilde{v}),$$
	which implies that
	\begin{align}\label{R3}
	2^{js}\|R_j^3\|_{L^p}&\leq C\sum_{\substack{|j-j'|\leq 4\\j''\leq j'-2}}2^{js}\|\Delta_{j''}\partial_{x} f\|_{L^\infty}\|\Delta_{j'}\widetilde{v}\|_{L^p}\notag \\
	&\leq C\sum_{\substack{|j-j'|\leq 4\\j''\leq j'-2}}2^{(j-j')s}2^{j''-j'}2^{j's}\|\Delta_{j''}f\|_{L^\infty}\|\Delta_{j'}\partial_{x}\widetilde{v}\|_{L^p}\notag \\
	&\leq C\|f\|_{L^\infty}\sum_{|j-j'|\leq 4}2^{j's}\|\Delta_{j'}\partial_{x}\widetilde{v}\|_{L^p}\notag \\
	&\leq Cc_j\|f\|_{L^\infty}\|\partial_x v\|_{B^s_{p,r}}.
	\end{align}
	Combining inequalities \eqref{Ri} and \eqref{R3} completes the proof.
\end{proof}
Now, we state a blow-up criterion for \eqref{eqn} as follows.

\begin{theo}\label{thm4.2}
	Let $1\leq p,r\leq \infty$, $s>\max(\frac{1}{2},\frac{1}{p})$ and $n_0\in B^s_{p,r}$. Let $T>0$ denote the maximal existence time of the corresponding solution $n$ to \eqref{eqn}. If $T$ is finite, then we have 
	\begin{equation}
	\int_0^T\|u(t)\|_{W^{1,\infty}}dt\int_{0}^T\|u(t)\|_{W^{1,\infty}}\|n(t)\|_{B^0_{\infty,\infty}}dt=\infty.
	\end{equation}  
\end{theo}
\begin{proof}
	For all $j\geq -1$, applying the operator $\Delta_j$ to the equation \ref{eqn}  yields
	\begin{equation}\label{eqnj}
	\left\{\begin{array}{l}
	\partial_t(\Delta_jn)+(u^{2}+u_{x}^{2})\partial_{x}(\Delta_jn)=R_j+\Delta_j G(u),\\
	(\Delta_jn)|_{t=0}=\Delta_jn_{0},
	\end{array}\right.
	\end{equation} 
	with $R_j\triangleq[(u^{2}+u_{x}^{2})\cdot\partial_{x},\Delta_j]n$.
	Multiplying both sides of \eqref{eqnj} by $sgn(\Delta_jn)|\Delta_jn|^{p-1}$ and integrating over $\mathbb{R}$, we have
	\begin{align*}
	\|\Delta_{j}n(t)\|_{L^{p}}&\leq \|\Delta_{j}n_{0}\|_{L^{p}}+\int_{0}^t\frac{1}{p}\left\|\partial_{x}(u^2+u_x^2)(t^{\prime})\right\|_{L^{\infty}}\left\|\Delta_{j}n(t^{\prime})\right\|_{L^{p}}dt'+\int_{0}^{t}\left\|\Delta_{j}G(t^{\prime})\right\|_{L^{p}}dt^{\prime}\\
	&~~~+\int_{0}^{t}\left\|R_{j}(t^{\prime})\right\|_{L^{p}}dt^{\prime}.
	\end{align*}
	Multiplying both sides of the above inequality by $2^{js}$ and taking the $\ell^r$-norm yields
	\begin{align}\label{n_bspr}
	\|n(t)\|_{B^s_{p,r}}&\leq \|n_{0}\|_{B^s_{p,r}}+C\int_{0}^t\|u_x(t')\|_{L^{\infty}}\|n(t')\|_{L^{\infty}}\|n(t^{\prime})\|_{B^s_{p,r}}dt'+\int_{0}^{t}\|G(t^{\prime})\|_{B^s_{p,r}}dt^{\prime}\notag\\
	&~~~+\int_{0}^{t}\|(2^{js}\|R_{j}(t^{\prime})\|_{L^{p}})_{j\geq -1}\|_{\ell^r}dt^{\prime}.
	\end{align}
	We now proceed to estimate the third and fourth terms on the right-hand side of \eqref{n_bspr}. Regarding the third term, we will use Bony's decomposition to obtain a more refined estimate than that in the proof of Theorem \ref{the3.1}. In fact, by using Lemmas \ref{bdd oper} and \ref{Moser},  together with the fact that $L^\infty \hookrightarrow B^0_{\infty,\infty}$, we have:
	\begin{align}\label{G_1}
	\|G_1(u)\|_{B^s_{p,r}}&= \|2u_xu_{xx}^2\|_{B^{s}_{p,r}}&\notag\\
	&\leq C\Big(\|u_xu_{xx}\|_{L^\infty}\|u_{xx}\|_{B^s_{p,r}}+\|u_{xx}\|_{L^\infty}\|u_xu_{xx}\|_{B^s_{p,r}}\Big)\notag\\
	&\leq C\Big(\|u_x\|_{L^\infty}\|n\|_{L^\infty}\|n\|_{B^s_{p,r}}+\|u_{xx}\|_{L^\infty}(\|u_x\|_{L^\infty}\|u_{xx}\|_{B^s_{p,r}}+\|u_{xx}\|_{B^{-1}_{\infty,\infty}}\|u_{x}\|_{B^{s+1}_{p,r}})\Big)\notag\\
	&\leq C\Big(\|u_x\|_{L^\infty}\|n\|_{L^\infty}\|n\|_{B^s_{p,r}}+\|u_{xx}\|_{L^\infty}(\|u_x\|_{L^\infty}\|n\|_{B^s_{p,r}}+\|u_{x}\|_{B^{0}_{\infty,\infty}}\|n\|_{B^{s}_{p,r}})\Big)\notag\\
	&\leq C\|u_x\|_{L^{\infty}}\|n\|_{L^{\infty}}\|n\|_{B^s_{p,r}},
	\end{align}
	\begin{align}\label{G_2}
	\|G_2(u)\|_{B^s_{p,r}}
	&=\|-\frac{5}{3}u^3-2uu_x^2-3u^2u_{xx}+16u_x^2u_{xx}-uu_{xx}^2\|_{B^{s}_{p,r}}\notag\\
	&\leq C\Big(\|u\|_{L^{\infty}}\|n\|_{L^{\infty}}\|n\|_{B^s_{p,r}} +(\|u_x\|_{L^{\infty}}\|uu_x\|_{B^s_{p,r}}+\|uu_x\|_{L^{\infty}}\|u_x\|_{B^s_{p,r}})\notag\\ 
	&~~~ +(\|u\|_{L^{\infty}}\|uu_{xx}\|_{B^s_{p,r}}+\|uu_{xx}\|_{L^{\infty}}\|u\|_{B^s_{p,r}})+(\|u_x\|_{L^{\infty}}\|u_xu_{xx}\|_{B^s_{p,r}}+\|u_xu_{xx}\|_{L^{\infty}}\notag\\ &~~~\times\|u_x\|_{B^s_{p,r}}) +(\|u_{xx}\|_{L^{\infty}}\|uu_{xx}\|_{B^s_{p,r}}+\|uu_{xx}\|_{L^{\infty}}\|u_{xx}\|_{B^s_{p,r}}) \Big)\notag\\
	&\leq C(\|u\|_{L^{\infty}}+\|u_x\|_{L^{\infty}})\|n\|_{L^{\infty}}\|n\|_{B^s_{p,r}},
	\end{align}
	and
	\begin{align}\label{G_3}
	\|G_3(u)\|_{B^s_{p,r}}&=\|-u_xu_{xx}^2-2uu_xu_{xx}\|_{B^{s}_{p,r}}\notag\\
	&\leq C\Big(\|u_x\|_{L^{\infty}}\|n\|_{L^{\infty}}\|n\|_{B^s_{p,r}} +(\|u_{x}\|_{L^{\infty}}\|uu_{xx}\|_{B^s_{p,r}}+\|uu_{xx}\|_{L^{\infty}}\|u_{x}\|_{B^s_{p,r}}) \Big)\notag\\
	&\leq C(\|u\|_{L^{\infty}}+\|u_x\|_{L^{\infty}})\|n\|_{L^{\infty}}\|n\|_{B^s_{p,r}}.
	\end{align}	
	Combining estimates \eqref{G_1}, \eqref{G_2} and \eqref{G_3} yields
	\begin{align}\label{G_est2}
	\|G(u)\|_{B^{s}_{p,r}}&\leq C\Big(\|G_1(u)\|_{B^{s}_{p,r}}+\|G_2(u)\|_{B^{s}_{p,r}}+\|G_3(u)\|_{B^{s}_{p,r}}\Big)\notag\\
	&\leq C\|u\|_{W^{1,\infty}}\|n\|_{L^{\infty}}\|n\|_{B^s_{p,r}}.
	\end{align}
	For the fourth term, applying Lemma \ref{lem4.1} with $v=u^2+u_x^2$ and using Bony's decomposition, we obtain
	\begin{align}\label{Rj_est}
	\|(2^{js}\|R_{j}\|_{L^{p}})_{j\geq -1}\|_{\ell^r}&\leq C\Big(\|\partial_x (u^2+u_x^2)\|_{L^\infty}\|n\|_{B^s_{p,r}}+\|\partial_x (u^2+u_x^2)\|_{B^s_{p,r}}\|n\|_{L^\infty}\Big)\notag\\
	&\leq C\Big(\|u_x\|_{L^{\infty}}\|n\|_{L^{\infty}}\|n\|_{B^s_{p,r}}+\|uu_x+u_xu_{xx}\|_{B^s_{p,r}}\|n\|_{L^{\infty}}\Big)\notag\\
	&\leq C\|u\|_{W^{1,\infty}}\|n\|_{L^{\infty}}\|n\|_{B^s_{p,r}}.
	\end{align}
	Plugging \eqref{G_est2} and \eqref{Rj_est} into \eqref{n_bspr}, we thus derive
	\begin{align}\label{n_bspr1}
	\|n(t)\|_{B^s_{p,r}}&\leq \|n_{0}\|_{B^s_{p,r}}+C\int_{0}^t\|u(t')\|_{W^{1,\infty}}\|n(t')\|_{L^{\infty}}\|n(t^{\prime})\|_{B^s_{p,r}}dt'.
	\end{align}
	Since $s > \frac{1}{p}$, we can choose $\epsilon \in (0, 1)$ with $s > \frac{1}{p} + \epsilon$. Applying the logarithmic interpolation inequality in Proposition \ref{basic} yields  
	\begin{align*}
	\|n\|_{L^\infty}&\leq C\|n\|_{B^0_{\infty,\infty}}\Big(1+\log\frac{\|n\|_{B^\epsilon_{\infty,\infty}}}{\|n\|_{B^0_{\infty,\infty}}}\Big)\leq C\Big( 1+\|n\|_{B^0_{\infty,\infty}}\log(e+\|n\|_{B^s_{p,r}}) \Big).
	\end{align*}
	Plugging the above estimate into \eqref{n_bspr1} and applying Gronwall's inequality yields
	\begin{align*}
	\|n(t)\|_{B^s_{p,r}}&\leq \|n_{0}\|_{B^s_{p,r}}\exp \Big(C\int_{0}^t\|u\|_{W^{1,\infty}}( 1+\|n\|_{B^0_{\infty,\infty}}\log(e+\|n\|_{B^s_{p,r}}))dt' \Big),
	\end{align*}
	which implies 
	\begin{align*}
	\log(e+\|n(t)\|_{B^s_{p,r}})&\leq \log(e+\|n_{0}\|_{B^s_{p,r}})+ C\int_{0}^t\|u\|_{W^{1,\infty}}( 1+\|n\|_{B^0_{\infty,\infty}}\log(e+\|n\|_{B^s_{p,r}}))dt'.
	\end{align*}
	Applying Gronwall's inequality once more, we obtain
	\begin{align}
	\log(e+\|n(t)\|_{B^s_{p,r}})&\leq \Big( \log(e+\|n_{0}\|_{B^s_{p,r}})+ C\int_{0}^t\|u\|_{W^{1,\infty}}dt'\Big)\exp\Big(C\int_{0}^t\|u\|_{W^{1,\infty}}\|n\|_{B^0_{\infty,\infty}}dt'\Big).
	\end{align}
	
	Therefore, we can conclude that if $$\int_{0}^T\|u(t)\|_{W^{1,\infty}}dt<\infty \quad \text{and} \quad  \int_0^T\|u(t)\|_{W^{1,\infty}}\|n(t)\|_{B^0_{\infty,\infty}}dt<\infty,$$
	we have $\|n\|_{L_T^\infty(B^s_{p,r})}<\infty$. Then, we infer that there exists $T_0$ which is small enough such that $4C^3\|n\|_{L_T^\infty (B^{s}_{p,r})}^2T_0 < 1$. We now choose $\delta<T_0$ such that $4C^3\|n(T-\delta)\|_{B^{s}_{p,r}}^2T_0 < 1$, then following same line of the proof of Theorem \ref{the3.1}, we can extend the solution $n$ on $[0,T)\times\mathbb{R}$ to $\widetilde{n}$ on $[0,T-\delta+T_0)\times\mathbb{R}$ which stands in contradiction to the definition of $T$.
\end{proof}

Noting that equation \eqref{eqn} possesses a conserved quantity $E(u)=\int_{\mathbb{R}}(u^2+2u_x^2+u_{xx}^2)dx$ in \cite{Li.Z}, we immediately obtain the following corollary of Theorem \ref{thm4.2}.

\begin{coro}\label{cor4.3}
	Under the hypotheses of Theorem \ref{thm4.2}, assume that $u_0 \in H^2$. Let $T>0$ denote the maximal existence time of the corresponding solution $n$ to \eqref{eqn}. If $T$ is finite, it follows that
	$$ \int_0^T\|n(t)\|_{B^0_{\infty,\infty}}dt=\infty.$$
	
\end{coro}

We now consider the equivalent equation \eqref{eq2} corresponding to \eqref{eqn}. Denote by $C_1$ a constant satisfying $C_1 < \widetilde{C}=(\frac{1}{4\pi})^{\frac{1}{4}}(\frac{\Gamma(\frac{3}{2})}{\Gamma(2)})^{\frac{1}{2}}=\frac{1}{2}$, where $\widetilde{C}$ is the sharp imbedding constant of Sobolev inequality which is given in \cite{Mo2001}. We will provide sufficient conditions for the blow-up data to the initial-value problem \eqref{eq2}. Our blow-up result can be stated as follows.
\begin{theo}\label{thm4.4}
	Let $1\leq p, r\leq \infty,~ s>\max(\frac{1}{2},\frac{1}{p}), ~n_0=(1-\partial_x^2)u_0\in B^s_{p,r}$ and $u_0\in H^2.$ Denote by $C$ the constant in \eqref{2}.
	Assume that there exists a point $x_0\in\mathbb{R}$ such that
	\begin{equation}\label{con}
	|u_{0x}(x_0)|\geq C_0\triangleq C_1\|u_0\|_{H^2}, ~
	u_{0x}(x_0)u_{0xx}(x_0)\leq -2\omega_0\sqrt{K}
	\end{equation}
	with $$T_1\triangleq\min\{\frac{C_0}{32\|u_0\|_{H^2}^3},\frac{1}{4C^3\|n_0\|_{B^s_{p,r}}^2}\},~  K\triangleq34\|u_0\|_{H^2}^4-\frac{1}{8}C_0^4+\frac{3025}{4C_0^2}\|u_0\|_{H^2}^6>0,~ \omega_0\triangleq1+\frac{2}{T_1\sqrt{K}}.$$
	Then, the solution $u$ will blow up within $[0,T_1]$.\par 
	 More precisely, $u$ will blow up at the time $T_2=\frac{1}{\sqrt{K}}\ln\frac{-2\sqrt{K}+u_{0x}(x_0)u_{0xx}(x_0)}{2\sqrt{K}+u_{0x}(x_0)u_{0xx}(x_0)}$. 
\end{theo}

\begin{proof}
	Define $q(t,x)=u_x(t,x)u_{xx}(t,x)$. Thanks to \eqref{eq2} and $(1-2\partial_x^2+\partial_x^4)P(D)=Id$, we note that $u_x$ and $u_{xx}$ satisfy 
	\begin{equation}\label{u_xt}
	u_{xt}+(u^{2}+u_{x}^{2})u_{xx}+2uu_x^2=\partial_xP(D)F_1+\partial_x^2P(D)F_2+\partial_x^3P(D)F_3,
	\end{equation}
	and
	\begin{equation}\label{u_xxt}
	u_{xxt}+(u^{2}+u_{x}^{2})u_{xxx}+2uu_xu_{xx}+2u_x^3+u_xu_{xx}^2=\partial_x^2P(D)F_1+\partial_x^3P(D)F_2-P(D)F_3+2\partial_x^2P(D)F_3.
	\end{equation}
	We thus have
	\begin{align*}
	q_t+q_x(u^2+u_x^2)&=u_{xt}u_{xx}+u_xu_{xxt}+(u_{xx}^2+u_xu_{xxx})(u^2+u_x^2) \notag\\
	&=-u_{x}^2u_{xx}^2-4uu_{x}^2u_{xx}-2u_x^4+u_{xx}(\partial_xP(D)F_1+\partial_x^2P(D)F_2+\partial_x^3P(D)F_3)\notag\\
	&~~~~~~+u_x(\partial_x^2P(D)F_1+\partial_x^3P(D)F_2-P(D)F_3+2\partial_x^2P(D)F_3),
	\end{align*}
	where $-4uu_{x}^2u_{xx}=-4(uu_{x})(u_{x}u_{xx})\leq 4(\frac{q^2}{8}+2u^2u_x^2)\leq\frac{1}{2}q^2+8\|u\|_{L^\infty}^2\|u_x\|_{L^\infty}^2$ which implies that
	\begin{align}\label{q1}
	q_t+q_x(u^2+u_x^2)&\leq -\frac{1}{2}q^2+8\|u\|_{L^\infty}^2\|u_x\|_{L^\infty}^2-2u_x^4+|u_{xx}|\|\partial_xP(D)F_1+\partial_x^2P(D)F_2+\partial_x^3P(D)F_3\|_{L^\infty}\notag\\
	&~~~~~~+\|u_x\|_{L^\infty}\|\partial_x^2P(D)F_1+\partial_x^3P(D)F_2-P(D)F_3+2\partial_x^2P(D)F_3\|_{L^\infty}.
	\end{align}
		
	Prior to the main argument, we derive several  estimates that will play a fundamental role in the forthcoming analysis:
	\begin{equation}\label{|u|}
	|u|=\Big(\int_{-\infty}^{x}uu_{x}dx-\int_{x}^{\infty}uu_{x}dx\Big)^{\frac{1}{2}} \leq\frac{\sqrt{2}}{2}\Big(\int_{-\infty}^{x}(u^{2}+u_{x}^{2})dx+\int_{x}^{\infty}(u^{2}+u_{x}^{2})dx\Big)^{\frac{1}{2}}=\frac{\sqrt{2}}{2}\|u\|_{H^1},
	\end{equation}
	\begin{equation}\label{|u_x|}
	|u_x|=\Big(\int_{-\infty}^{x}u_xu_{xx}dx-\int_{x}^{\infty}u_xu_{xx}dx\Big)^{\frac{1}{2}} \leq\frac{\sqrt{2}}{2}\Big(\int_{-\infty}^{x}(u_x^{2}+u_{xx}^{2})dx+\int_{x}^{\infty}(u_x^{2}+u_{xx}^{2})dx\Big)^{\frac{1}{2}}\leq\frac{\sqrt{2}}{2}\|u\|_{H^2}.
	\end{equation}
	Through computation of $\|\partial_x^i(\frac{1}{4}e^{-|\cdot|}(1+|\cdot |))\|_{L^1}$ and $\|\partial_x^i(\frac{1}{4}e^{-|\cdot|}(1+|\cdot |))\|_{L^\infty}$ for $i\in\{0,1,2,3\}$ and using \eqref{|u|}, \eqref{|u_x|}, we derive the following estimates:
	\begin{equation}\label{eq4.11}
	\|\partial_xP(D)F_1 \|_{L^\infty}\leq \|\partial_x(\frac{1}{4}e^{-|\cdot|}(1+|\cdot |))\|_{L^1}\|F_1\|_{L^\infty}\leq \frac{1}{6}\|u_x\|_{L^\infty}^3\leq \frac{\sqrt{2}}{24}\|u\|_{H^2}^3,
	\end{equation}
	\begin{align}
	\|\partial_x^2P(D)F_2 \|_{L^\infty}&\leq \|\partial_x^2(\frac{1}{4}e^{-|\cdot|}(1+|\cdot |))\|_{L^1}\|-\frac{5}{3}u^3-5uu_x^2\|_{L^\infty} +\|\partial_x^2(\frac{1}{4}e^{-|\cdot|}(1+|\cdot |))\|_{L^\infty}\notag\\
	&~\times\|-3u^2u_{xx}+24u_x^2u_{xx}-uu_{xx}^2\|_{L^1} \notag\\
	&\leq \frac{5\sqrt{2}}{6}\|u\|_{H^2}^3+\frac{1}{4}(\frac{3}{2}\|u\|_{L^\infty}+12\|u_x\|_{L^\infty}+\|u\|_{L^\infty})\int_{\mathbb{R}}(u^2+u_x^2+u_{xx}^2)dx\notag\\
	&\leq \frac{127\sqrt{2}}{48}\|u\|_{H^2}^3, 
	\end{align}
	\begin{equation}
	\|\partial_x^3P(D)F_3\|_{L^\infty}\leq \|\partial_x^3(\frac{1}{4}e^{-|\cdot|}(1+|\cdot |))\|_{L^\infty}\|u_xu_{xx}^2+4uu_xu_{xx}\|_{L^1}
	\leq \frac{3\sqrt{2}}{4}\|u\|_{H^2}^3,
	\end{equation}
	\begin{equation}
	\|\partial_x^2P(D)F_1 \|_{L^\infty}\leq \|\partial_x^2(\frac{1}{4}e^{-|\cdot|}(1+|\cdot |))\|_{L^1}\|\frac{1}{3}u_x^3\|_{L^\infty}
	\leq \frac{\sqrt{2}}{24}\|u\|_{H^2}^3,
	\end{equation}
	\begin{align}
	\|\partial_x^3P(D)F_2 \|_{L^\infty}&\leq \|\partial_x^3(\frac{1}{4}e^{-|\cdot|}(1+|\cdot |))\|_{L^1}\|-\frac{5}{3}u^3-5uu_x^2\|_{L^\infty} +\|\partial_x^3(\frac{1}{4}e^{-|\cdot|}(1+|\cdot |))\|_{L^\infty}\notag\\
	&~~\times\|-3u^2u_{xx}+24u_x^2u_{xx}-uu_{xx}^2\|_{L^1} \notag\\
	&\leq \frac{127\sqrt{2}}{24}\|u\|_{H^2}^3, 
	\end{align}
	\begin{equation}
	\|P(D)F_3\|_{L^\infty}\leq \|\frac{1}{4}e^{-|\cdot|}(1+|\cdot |)\|_{L^\infty}\|u_xu_{xx}^2+4uu_xu_{xx}\|_{L^1}
	\leq \frac{3\sqrt{2}}{8}\|u\|_{H^2}^3,
	\end{equation}
	\begin{equation}\label{eq4.17}
	\|2\partial_x^2P(D)F_3\|_{L^\infty}\leq 2\|\partial_x^2(\frac{1}{4}e^{-|\cdot|}(1+|\cdot |))\|_{L^\infty}\|u_xu_{xx}^2+4uu_xu_{xx}\|_{L^1} 
	\leq \frac{3\sqrt{2}}{4}\|u\|_{H^2}^3.
	\end{equation}
    Plugging \eqref{|u_x|}$\textendash$\eqref{eq4.17} into  \eqref{q1}, we obtain  
	\begin{equation}
	q_t+q_x(u^2+u_x^2)\leq -\frac{1}{2}q^2-2u_x^4+\frac{203}{24}\|u\|_{H^2}^4+\frac{55\sqrt{2}}{16}\|u\|_{H^2}^3|u_{xx}|.
	\end{equation}
	From \cite{Li.Z}, we know that equation \eqref{eq2} satisfies the conservation law: 
	$$E(u)=\int_{\mathbb{R}}(u^2+2u_x^2+u_{xx}^2)dx=E(u_0). $$
	Then, we have
	\begin{equation*}
	\|u\|_{H^2}^2\leq E(u)=E(u_0)\leq 2\|u_0\|_{H^2}^2,
	\end{equation*}
	which yields, for all $t>0$,  
	\begin{equation}\label{u H^2}
	\|u(t)\|_{H^2}\leq \sqrt{2}\|u_0\|_{H^2}.
	\end{equation}
	At the same time, for $|u_{xx}|$, setting $\epsilon=\frac{2\sqrt{2}}{55\|u\|_{H^2}^3} $, we have
	\begin{equation}\label{|u_xx|}
	|u_{xx}|=\frac{|q|}{|u_x|}\leq \epsilon q^2+\frac{1}{4\epsilon |u_x|^2}\leq \frac{2\sqrt{2}}{55\|u\|_{H^2}^3}q^2+ \frac{55\sqrt{2}\|u\|_{H^2}^3}{16|u_x|^2}.
	\end{equation}
    Plugging \eqref{u H^2} and \eqref{|u_xx|} into \eqref{q1} yields
    \begin{equation*}
    q_t+q_x(u^2+u_x^2)\leq -\frac{1}{4}q^2-2u_x^4+\frac{203}{6}\|u_0\|_{H^2}^4+\frac{3025}{16}\|u_0\|_{H^2}^6\frac{1}{u_x^2},
    \end{equation*}
	which implies that in Lagrangian coordinates, we have 
	\begin{equation}\label{q2}
	q_t(t,y(t,x_0))\leq -\frac{1}{4}(q(t,y(t,x_0)))^2-2(u_x(t,y(t,x_0)))^4+\frac{203}{6}\|u_0\|_{H^2}^4+\frac{3025}{16}\|u_0\|_{H^2}^6\frac{1}{(u_x(t,y(t,x_0)))^2},
	\end{equation}
	where the flow map $y(t,x_0)$ satisfies \eqref{eqflow} at $x=x_0$.\par
	Under the condition \eqref{con}, we derive $|u_{x}(t,y(t,x_0))|\geq \frac{1}{2}C_0$ for all $t\in[0,T_1]$. Indeed, using \eqref{u_xt}, we get 
	\begin{align}
	|\partial_t(u_{x}(t,y(t,x_0)))|&=|-2(uu_x^2)(t,y(t,x_0))+(\partial_xP(D)F_1+\partial_x^2P(D)F_2+\partial_x^3P(D)F_3)(t,y(t,x_0))|\notag\\
	&\leq 2\|u\|_{L^\infty}\|u_x\|_{L^\infty}^2+\|\partial_xP(D)F_1+\partial_x^2P(D)F_2+\partial_x^3P(D)F_3\|_{L^\infty}\notag\\
	&\leq 16\|u_0\|_{H^2}^3.
	\end{align} 
	Since $u_{x}(t,y(t,x))=\int_{0}^t \partial_\tau (u_{x}(\tau,y(\tau,x)))d\tau+u_{0x}(x)$, we thus have, for all $t\in[0,T_1]$
	\begin{align}
	|u_{x}(t,y(t,x_0))|&\geq |u_{0x}(x_0)|-|\int_{0}^t \partial_\tau (u_{x}(\tau,y(\tau,x)))d\tau| \notag\\
	&\geq C_0-\int_{0}^t |\partial_\tau (u_{x}(\tau,y(\tau,x)))|d\tau\notag\\
	&\geq C_0-16\|u_0\|_{H^2}^3T_1\notag\\
	&\geq \frac{1}{2}C_0.
	\end{align}
	Inserting the above inequality into \eqref{q2}, we obtain
	\begin{align}\label{q3}
	q_t(t,y(t,x_0))&\leq -\frac{1}{4}(q(t,y(t,x_0)))^2-\frac{1}{8}C_0^4+\frac{203}{6}\|u_0\|_{H^2}^4+\frac{3025}{4C_0^2}\|u_0\|_{H^2}^6 \notag\\
	&\leq -\frac{1}{4}(q(t,y(t,x_0)))^2+K.
	\end{align} 
	Owing to the condition \eqref{con}, $-\frac{1}{4}q_0^2(x_0)+K<0$ holds.
	We can easily deduce that $q(t,y(t,x_0))<-2\sqrt{K}$ for all $t\in[0,T_1]$ by \eqref{q3}.\par 
	Consider the ordinary differential equation
	$$f'(t)=-\frac{1}{4}f^2(t)+K,~f(0)=q_0(x_0)<-2\sqrt{K}, $$
	then we have 
	\begin{align*}
	q(t,y(t,x_0))\leq f(t)=\frac{2\sqrt{K}(1-Ce^{-\sqrt{K}t})}{1+Ce^{-\sqrt{K}t}}=-2\sqrt{K}+\frac{4\sqrt{K}}{1+Ce^{-\sqrt{K}t}}
	\end{align*}
	with $C=\frac{2\sqrt{K}-q_0(x_0)}{2\sqrt{K}+q_0(x_0)}<0.$
	Observing that $t\mapsto 1+Ce^{-\sqrt{K}t}$ is an increasing function and $$1+Ce^{-\sqrt{K}t}\Big|_{t=0}=1+C=\frac{4\sqrt{K}}{2\sqrt{K}+q_0(x_0)}<0, ~1+Ce^{-\sqrt{K}T_2}=0$$
	we infer that 
	$$\lim_{t\rightarrow T_2^{-}} \inf_{x\in\mathbb{R}} (u_x(x,t)u_{xx}(x,t))\leq \lim_{t\rightarrow T_2^{-}}q(t,y(t,x_0))=-\infty.$$
	Finally, since $$T_2=\frac{1}{\sqrt{K}}\ln\frac{-2\sqrt{K}+q_0(x_0)}{2\sqrt{K}+q_0(x_0)}=\frac{1}{\sqrt{K}}\ln(1+\frac{4\sqrt{K}}{-(2\sqrt{K}+q_0(x_0))})\leq \frac{1}{\sqrt{K}}\frac{4\sqrt{K}}{-(2\sqrt{K}+q_0(x_0))}\leq T_1,$$
	the proof is complete.
	
\end{proof}
\section{Ill-posedness}
In this section, we will establish the ill-posedness of equation \eqref{eqn} in the critical Sobolev space $H^{\frac{1}{2}}$ by proving the occurrence of norm inflation. Our ill-posedness result is stated as follows:
\begin{theo}\label{thm5.1}
	 For any sufficiently large positive integer $N\in\mathbb{N}^+$, there exists an initial data $n_0^N\in H^\infty(\mathbb{R})$ such that the following conclusions hold:
	\begin{enumerate}[(1)]
		\item $\|n_0^N\|_{H^{\frac{1}{2}}}\lesssim\frac{1}{\ln N}$;
		\item There exists a unique solution $n^N\in C([0,T);H^\infty(\mathbb{R}))$ to  equation \eqref{eqn} with initial data $n_0^N$, and the maximal lifespan $T^*$ satisfying $T^*<N^{-\frac{1}{6}}$;
		\item $\limsup\limits_{t\rightarrow T^*}\|n^N\|_{H^{\frac{1}{2}}}\gtrsim  \limsup\limits_{t\rightarrow T^*}\|n^N\|_{B^{0}_{\infty,\infty}}=+\infty.$
	\end{enumerate}
	
\end{theo}

\begin{proof}
	Let $\psi\in C_0^\infty(\mathbb{R})$ is an even function with $\mathrm{Supp}\  \psi\subset\{\xi\in\mathbb{R}:\frac{4}{3}\leq |\xi|\leq \frac{3}{2}\}$ and values in $[0,1]$. Let $g\in\mathcal{S}$ be an odd function such that
	$$\frac{g'(0)}{1+\|g\|_{H^2}}\geq\frac{C_0}{4},$$ 
	where the constant $C_0$ will be specified later. Define 
	$$ u_0^N(x)=\frac{1}{\ln N}\sum_{j=1}^N\frac{1}{2^{3j}j^{\frac{2}{3}}}f_j(x)+\frac{1}{\ln N}g(x),$$
	with $\widehat{f_j}(\xi)=\psi(2^{-j}\xi)$. A straightforward computation shows that $n_0^N=(1-\partial_{x}^2)u_0^N\in H^\infty(\mathbb{R})$ and that $f_j$ is an even function satisfying $\Delta_l f_j = \delta_{lj} f_j$ for all $l \leq -1$, where $\delta_{lj} $ denotes the Kronecker delta. This implies that, for $1\leq l\leq N$, 
	\begin{align*}
	\|\Delta_lu_0^N\|_{L^2}&=\|\frac{1}{\ln N}\mathcal{F}^{-1}\Big(\varphi(2^{-l}\xi)\sum_{j=1}^N\frac{1}{2^{3j}j^{\frac{2}{3}}}\psi(2^{-j}\xi)\Big)+\frac{1}{\ln N}\Delta_l g\|_{L^2}\notag\\
	&=\|\frac{f_l}{2^{3l}l^{\frac{2}{3}}\ln N}+\frac{1}{\ln N}\Delta_l g\|_{L^2}\notag\\
	&\leq \frac{C}{2^{\frac{5l}{2}}l^{\frac{2}{3}}\ln N}+\frac{1}{\ln N}\|\Delta_l g\|_{L^2}.
	\end{align*}
	We thus deduce that
	\begin{align}\label{u0N_H2}
	\|u_0^N\|_{H^2}\leq C\|u_0^N\|_{B^2_{2,2}} 
	=C\Big(\sum_{l\geq -1}2^{4l}\|\Delta_lu_0^N\|_{L^2}^2\Big)^{\frac{1}{2}}
	&\leq \frac{C}{\ln N}\Big(\Big(\sum_{l\geq 1}\frac{1}{2^{l}l^{\frac{4}{3}}}\Big)^{\frac{1}{2}}+\|g\|_{H^2}\Big)\notag\\
	&\leq \frac{C_0}{\ln N}\Big(1+\|g\|_{H^2}\Big),
	\end{align}
	where the constant $C_0$ depends only on $d$ and $\|\psi\|_{L^2}$. At the same time, we have
	\begin{align*}
	\|u_0^N\|_{B^{\frac{5}{2}}_{2,2}}=\Big(\sum_{l\geq -1}2^{5l}\|\Delta_lu_0^N\|_{L^2}^2\Big)^{\frac{1}{2}} 
	\lesssim\frac{1}{\ln N}\Big(\Big(\sum_{l\geq 1}\frac{1}{l^{\frac{4}{3}}}\Big)^{\frac{1}{2}}+\|g\|_{H^{\frac{5}{2}}}\Big)
	\lesssim \frac{1}{\ln N},
	\end{align*}
	which implies that
	\begin{equation}\label{n0N}
	\|n_0^N\|_{H^{\frac{1}{2}}}=\|(1-\partial_{x}^2)u_0^N\|_{H^{\frac{1}{2}}}\lesssim \|u_0^N\|_{B^{\frac{5}{2}}_{2,2}}\lesssim \frac{1}{\ln N}.
	\end{equation}
	
	Now, we proceed to calculate $u_{0x}^N$ and $u_{0xx}^N$ at $x=0$. According to \eqref{u0N_H2}, we have
	\begin{align}\label{u0x(0)}
	u_{0x}^N(0)=\frac{1}{\ln N}\sum_{j=1}^N\frac{1}{2^{3j}j^{\frac{2}{3}}}f_j'(0)+\frac{1}{\ln N}g'(0)\geq\frac{C_0}{4\ln N}\Big(1+\|g\|_{H^2}\Big) \geq \frac{1}{4}\|u_0^N\|_{H^2}
	\end{align}
	and
	\begin{align*}
	u_{0xx}^N(0)=\frac{1}{\ln N}\sum_{j=1}^N\frac{1}{2^{3j}j^{\frac{2}{3}}}f_j''(0)+\frac{1}{\ln N}g''(0) &=\frac{1}{\ln N} \sum_{j=1}^N\frac{1}{2^{3j}j^{\frac{2}{3}}}\mathcal{F}^{-1}(-\xi^2\psi(2^{-j}\xi))(0)\notag\\
	&=-\frac{C}{\ln N}\sum_{j=1}^N\frac{1}{j^{\frac{2}{3}}}\notag\\
	&\sim-\frac{N^\frac{1}{3}}{\ln N},
	\end{align*}
	which yield 
	\begin{align}\label{u0xu0xx(0)}
	u_{0x}^N(0)u_{0xx}^N(0)\sim-\frac{N^\frac{1}{3}}{(\ln N)^2}\rightarrow -\infty,\quad \text{as}~ N\rightarrow\infty.
	\end{align}
	Since $n_0^N\in H^\infty(\mathbb{R})$, Theorem \ref{the3.1} yields that there exists a unique solution $n^N\in C([0,T);H^\infty(\mathbb{R}))$ to  equation \eqref{eqn} with initial data $n_0^N$. Subsequently, by virtue of \eqref{u0N_H2}, \eqref{n0N}, \eqref{u0x(0)} and \eqref{u0xu0xx(0)} and applying Theorem \ref{thm4.4}, we can conclude that for sufficiently large 
	$N$, the maximal lifespan $T^*$ of $n^N$ satisfies 
	\begin{align*}
	T^*\leq \frac{1}{\sqrt{K}}\ln\frac{-2\sqrt{K}+u^N_{0x}(0)u^N_{0xx}(0)}{2\sqrt{K}+u^N_{0x}(0)u^N_{0xx}(0)}\sim(\ln N)^2\ln \frac{N^\frac{1}{3}+\frac{1}{\ln N}}{N^\frac{1}{3}-\frac{1}{\ln N}}\sim\frac{\ln N}{N^\frac{1}{3}}<N^{-\frac{1}{6}}.
	\end{align*}
	
	Finally, thanks to Corollary \ref{cor4.3}, we have 
	$$\sup_{t\in[0,T^*)}\|n^N\|_{H^{\frac{1}{2}}}\gtrsim \sup_{t\in[0,T^*)}\|n^N\|_{B^{0}_{\infty,\infty}}=+\infty. $$
	This completes the proof of Theorem \ref{thm5.1}.  
\end{proof}

	\smallskip
	\noindent\textbf{Acknowledgments.}~~This work was
	supported by the National Natural Science Foundation of China (No.12171493).

	\phantomsection
	\addcontentsline{toc}{section}{\refname}
	\bibliographystyle{abbrv} 
	\bibliography{refref}

\end{document}